\documentclass[a4paper]{article}%
\usepackage{amsmath}
\usepackage{amsfonts}
\usepackage{amssymb}
\usepackage{amsthm}
\usepackage[dvipsnames]{xcolor}
\usepackage{graphicx}
\usepackage{tikz}
\usetikzlibrary{snakes, arrows}

\usetikzlibrary{positioning}
\usetikzlibrary{calc}
\usepackage{subfigure}
\usepackage{pgfplots}
\usepackage{etex}
\usepackage{authblk}
\usepackage{enumerate}
\usepackage{caption}
\usepackage[]{algorithm2e}

\definecolor{dbrown}{RGB}{102,51,0}
\definecolor{darkblue}{RGB}{80, 100, 150}

\newtheorem{theorem}{Theorem}
\numberwithin{theorem}{section}

\newtheorem{corollary}[theorem]{Corollary}

\newcommand{\ie}{{\it i.e.},\ }

\begin{document}

\title{Inhomogeneous perturbation and error bounds for the stationary performance of random walks in the quarter plane}
\author{Xinwei Bai \and
		Jasper Goseling}
\date{}
\maketitle

\begin{abstract}
A continuous-time random walk in the quarter plane with homogeneous transition rates is considered. Given a non-negative reward function on the state space, we are interested in the expected stationary performance. Since a direct derivation of the stationary probability distribution is not available in general, the performance is approximated by a perturbed random walk, whose transition rates on the boundaries are changed such that its stationary probability distribution is known in closed form. 

A perturbed random walk for which the stationary distribution is a sum of geometric terms is considered and the perturbed transition rates are allowed to be inhomogeneous. It is demonstrated that such rates can be constructed for any sum of geometric terms that satisfies the balance equations in the interior of the state space. The inhomogeneous transitions relax the pairwise-coupled structure on these geometric terms that would be imposed if only homogeneous transitions are used.

An explicit expression for the approximation error bound is obtained using the Markov reward approach, which does not depend on the values of the inhomogeneous rates but only on the parameters of the geometric terms. Numerical experiments indicate that inhomogeneous perturbation can give smaller error bounds than homogeneous perturbation. 

\end{abstract}

\section{Introduction}
\label{sec:introduction}


In this paper, a continuous-time random walk $R$ in the two-dimensional non-negative orthant is considered, \ie the state space $S = \{ 0,1,2,\dots \}^2$. The transition rates of $R$ are homogeneous, which means that they are translation invariant within the interior, the horizontal axis and the vertical axis of the state space. The stationary performance of $R$ is studied in this paper. More precisely, given a non-negative reward function on the state space, $F: S \rightarrow [0,\infty)$, we are interested in the expected stationary reward given by
\begin{eqnarray*}
\mathcal{F} = \sum_{(n_1,n_2)\in S} \pi(n_1,n_2) F(n_1,n_2),
\end{eqnarray*}
where $\pi$ is the stationary probability distribution of $R$. In line with the work in~\cite{chen2015invariant},~\cite{goseling2016linear} and~\cite{vandijk11inbook}, upper and lower bounds on $\mathcal{F}$ are considered, which are obtained by considering a perturbed random walk $\bar{R}$ and its expected stationary reward,
\begin{eqnarray*}
\bar{\mathcal{F}} = \sum_{(n_1,n_2)\in S} \bar{\pi}(n_1,n_2) F(n_1,n_2).
\end{eqnarray*}
The transition rates of $\bar{R}$ are different from those of $R$ in such a way that $\bar{\pi}$ is known in closed form. As a result, upper and lower bounds on $\mathcal{F}$ can be established in terms of $\bar{\pi}$. 

The main difference from the paper mentioned above is that for the perturbed random walk, we consider inhomogeneous perturbation. More precisely, the transition rates of $\bar{R}$ are different at every state on the horizontal and vertical axes, which are constructed such that a given $\bar{\pi}$ satisfies all the balance equations. In this paper, $\bar{\pi}$ is considered to have the following form,
\begin{eqnarray*}
\bar{\pi}(n_1,n_2) = \sum_{k=1}^K c_k \rho_k^{n_1} \sigma_k^{n_2}, \qquad \forall\ (n_1,n_2)\in S,
\end{eqnarray*}
where $(\rho_k,\sigma_k) \in (0,1)^2$ and $c_k\in \mathbb{R}$. 

The approximation framework that is developed in this paper builds on the Markov reward approach, which is developed by van Dijk. An overview of this method is given in~\cite{vandijk11inbook}. In the main result of the approach, the bound on $| \bar{\mathcal{F}} - \mathcal{F} |$, which will be referred to as the error bound throughout the paper, depends on $\bar{\pi}$, the difference between the transition rates of $R$ and $\bar{R}$, and the so-called bias terms, whose definition will be given later. In most cases, no closed-form expressions are known for the bias terms. Therefore, the general approach in~\cite{vandijk1988simple},~\cite{vandijk1998bounds} and~\cite{vandijk1988perturbation} has been developed to establish bounds on the bias terms. These bounds are then used to establish the error bound. In this paper, the focus is not on finding such bounds on the bias terms. Instead, we assume that affine bounds on the bias terms are given. Based on these bounds, the error bound is derived for inhomogeneous perturbation. 

In particular, in~\cite{goseling2016linear} a general methodology has been proposed that provides affine bounds on the bias terms by means of linear programming. This approach not only gets rid of the complicated manual derivation, but also finds tighter bounds on the bias terms than the general approach. Hence, the error bound result has been improved. Moreover, the given $\bar{\pi}$ has a product form, \ie
\begin{eqnarray*}
\bar{\pi}(n_1,n_2) = (1-\rho)(1-\sigma) \rho^{n_1}\sigma^{n_2}, \qquad \forall\ (n_1,n_2)\in S,
\end{eqnarray*}
where $(\rho,\sigma) \in (0,1)^2$. Considering $\bar{\pi}$ of a product form often leads to relatively large difference between the transition rates of $R$ and $\bar{R}$. Taking that into account, the linear programming approach and homogeneous perturbation have been applied for discrete-time random walks in~\cite{chen2015invariant}, where  $\bar{\pi}(n)$ is a sum of geometric terms, \ie 
\begin{eqnarray*}
\bar{\pi}(n_1,n_2) = \sum_{k=1}^K c_k \rho_k^{n_1} \sigma_k^{n_2}, \qquad \forall\ (n_1,n_2)\in S.
\end{eqnarray*}
It is shown by numerical experiments that the error bound provided by considering the sum of geometric terms is better than that obtained using the product-form distribution (a single term). The reason is that with more geometric terms, the difference between transition rates of $R$ and $\bar{R}$ becomes smaller. 

In~\cite{chen2015invariant} and~\cite{goseling2016linear}, homogeneous perturbation is considered, which imposes constraints on $\bar{R}$ and $\bar{\pi}$. Necessary conditions on the structure of the geometric terms are given in~\cite{chen2013necessary}. More precisely, the transition rates in the interior of the state space, on the horizontal axis and on the vertical axis each determine an algebraic curve in $\mathbb{R}^2$, which is given by
{\small
\begin{align*}
Q &= \left\{ (\rho,\sigma) \in (0,1)^2 \mid \sum_{i=-1}^1\sum_{j=-1}^1\rho^{-i}\sigma^{-j}q_{i,j} = 0 \right\}, \\
H &= \left\{ (\rho,\sigma) \in (0,1)^2 \mid \sum_{i=-1}^1\rho^{-i}\sigma q_{i,-1} + \rho^{-1} h_{1,0} + \rho h_{-1,0} = \sum_{i=-1}^1 h_{i,1} + h_{-1,0} + h_{1,0} \right\}, \\
V &= \left\{ (\rho,\sigma) \in (0,1)^2 \mid \sum_{j=-1}^1\rho\sigma^{-j} q_{-1,j} + \sigma^{-1} v_{0,1} + \sigma v_{0,-1} = \sum_{j=-1}^1 v_{1,j} + v_{0,-1} + v_{0,1} \right\}. 
\end{align*}}
Every pair $(\rho_k,\sigma_k)$ in the sum has to be located on the curve induced by the interior balance equation. Besides, a pairwise-coupled structure has to be satisfied by $(\rho_1,\sigma_1),\dots,(\rho_K,\sigma_K)$, which means that for any two adjacent pairs $(\rho_k, \sigma_k)$ and $(\rho_{k+1}, \sigma_{k+1})$, it holds that $\rho_k = \rho_{k+1}$ or $\sigma_k = \sigma_{k+1}$. One of the contributions of the paper is to show that by allowing for inhomogeneous perturbation, these constraints are no longer necessary. 

Adan, Wessels and Zijm have developed a compensation approach in~\cite{adan1993compensation} to construct a sum of infinitely many geometric terms as the stationary probability distribution for continuous-time random walks. However, to ensure the convergence of the sum, it is required that the random walk has no transitions to the east, the north or the northeast, which puts a limitation on the models that can be considered. In this paper, we are going to consider random walks with transitions to the east, north, or northeast, for which the compensation approach can not be applied. Moreover, in our model, $(\rho_1,\sigma_1),\dots, (\rho_K,\sigma_K)$ don't necessarily follow a pairwise-coupled structure. In other words, with our framework we generalize the range of random walks whose stationary performance can be approximated. 

Works have been considered and various approximation frameworks have been applied on the same model. In~\cite{fayolle1999random}, boundary value problems are formulated for the generating functions of the stationary probability distributions. However, the problem can be solved only in a limited number of cases. Miyazawa considered the tail asymptotics of the stationary probability distributions for two-dimensional reflecting processes in~\cite{miyazawa2011light}, where an overview on the approaches to tackle the tail problem is given and their applications are discussed. Perturbation analysis was also considered in, for example,~\cite{altman2004perturbation},~\cite{haviv1984perturbation},~\cite{heidergott2010series} and~\cite{heidergott2007series}. Perturbation on quasi birth-and-death processes was discussed in~\cite{altman2004perturbation}, where an explicit expression for the stationary probability distribution of the perturbed process was given in terms of the perturbation parameter. However, no result on the error bound was given there. In~\cite{heidergott2010series} and~\cite{heidergott2007series}, the stationary probability distribution of the perturbed process was expressed in terms of the generator and the deviation matrix for discrete-time and continuous-time Markov processes respectively. Moreover, the condition for the existence of the deviation matrix is given there. Similar perturbation was considered and error bounds were given in~\cite{haviv1984perturbation}. Compared to the previous works, the main contributions of this paper are as follows. 

\begin{itemize}
\item It is shown that for any given $\bar{\pi}$ that is a sum of geometric terms, with the parameters located on the curve induced by the interior balance equation, show that inhomogeneous transition rates on the horizontal and vertical axes of $\bar{R}$ can be constructed such that $\bar{\pi}$ is the stationary probability distribution of $\bar{R}$. 

\item Based on the construction for the transition rates of $\bar{R}$, an explicit expression is given for the error bound. This expression depends only on the given parameters of $\bar{\pi}$, $R$ and $\bar{R}$, instead of the values of the inhomogeneous transition rates, which means that detailed calculation for the inhomogeneous rates is not necessary. 

\item Numerical results indicate that the error bound gets improved compared with the one obtained by homogeneous perturbation. The reason is that by allowing inhomogeneous perturbation, the difference between transition rates of $R$ and $\bar{R}$ becomes smaller.
\end{itemize}

The rest of the paper is structured as follows. In Section~\ref{sec:model_and_problem_formulation}, the random walk model and the problems to be solved are given. Next, the inhomogeneous perturbation framework is presented in Section~\ref{sec:inhomo_perturbed}, where the transition rates of the perturbed random walk are constructed. Moreover, the feasibility of our proposed approach for constructing the transition rates is shown. Then, in Section~\ref{sec:error_bound}, the error bound result is derived. Finally, numerical experiments are considered in Section~\ref{sec:numerical_experiment}, whose results suggest that our inhomogeneous perturbation framework gives smaller error bound than homogeneous perturbation.

\section{Model and problem formulation}
\label{sec:model_and_problem_formulation}

Consider a continuous-time random walk $R$, on the two-dimensional non-negative orthant. The state space is $S = \left\{ 0,1,2,\dots \right\}^2$. 
A state is represented by a two-dimensional vector $(n_1,n_2)$. Moreover, only transitions to the nearest neighbors of a state are allowed in the model. Let $N(n_1,n_2)$ be the set of all possible nonzero transitions from $(n_1,n_2)$, \ie 
\begin{eqnarray*}
N(n_1,n_2) = \left\{ (i,j)\in \{-1,0,1\}^2 \mid (i,j)\neq (0,0), (n_1+i,n_2+j) \in S \right\}. 
\end{eqnarray*}
Partition the state space into four disjoint components, namely the horizontal axis $S_1 = \{ 1,2,\dots \} \times \{ 0 \}$, the vertical axis $S_2 = \{ 0 \} \times \{ 1,2,\dots \}$, the origin $S_3 = \{ 0 \} \times \{ 0 \}$, and the interior $S_4 = \{ 1,2,\dots \} \times \{ 1,2,\dots \}$. Then, all the states in a component have the same set of possible transitions. For simplicity of notation, define the possible transitions in all the components as $N_1 = \{\{-1,0,1\}\times\{0,1\}\}\backslash \{(0,0)\}$, $N_2 = \{\{0,1\}\times\{-1,0,1\}\}\backslash \{(0,0)\}$, $N_3 = \{\{0,1\}\times\{0,1\}\}\backslash \{(0,0)\}$ and $N_4 = \{\{-1,0,1\}\times\{-1,0,1\}\}\backslash \{(0,0)\}$. Let $q_{i,j}(n_1,n_2)$ denote the transition rate from $(n_1,n_2)$ to $(n_1+i,n_2+j)$, for which $(i,j) \in N(n_1,n_2)$, and define
\begin{eqnarray*}
q_{0,0}(n_1,n_2) = - \sum_{(i,j) \in N(n_1,n_2)} q_{i,j}(n_1,n_2). 
\end{eqnarray*}
In the paper, we consider uniformizable random walks, \ie there exists $0 < \gamma < \infty$, for which 
\begin{eqnarray*}
\sum_{(i,j) \in N(n_1,n_2)} q_{i,j}(n_1,n_2) \le \gamma, \qquad \forall (n_1,n_2) \in S. 
\end{eqnarray*}
Homogeneous transition rates in each component are considered, which means that they are independent of the states within the component. Hence, let
\begin{eqnarray*}
q_{i,j}(n_1,n_2) = \left\{
\begin{array}{l@{\qquad }l}
h_{i,j}, & \textrm{if } (n_1,n_2) \in S_1, \\
v_{i,j}, & \textrm{if } (n_1,n_2) \in S_2, \\
r_{i,j}, & \textrm{if } (n_1,n_2) \in S_3, \\
q_{i,j}, & \textrm{if } (n_1,n_2) \in S_4. 
\end{array}
\right.
\end{eqnarray*}
To give an intuition on the model, the components and possible transitions are shown in Figure~\ref{fig:model_2d}. 

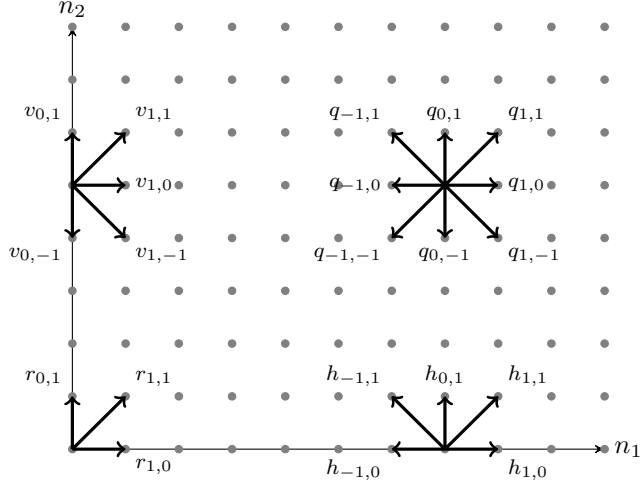
\begin{figure}[htb!]
\centering
\begin{tikzpicture}[scale = 0.7]
\draw [->, >=stealth'] (0,0) -- (0,8) node[above, thick] {$n_2$};
\draw [->, >=stealth'] (0,0) -- (10,0) node[right, thick] {$n_1$};

\foreach \i in {0,...,10}
\foreach \j in {0,...,8}
\filldraw [gray] (\i, \j) circle (2pt);
	
\draw [->, very thick] (7,5) -- ++(-1,-1) node[below left, thick] {\footnotesize $q_{-1,-1}$};
\draw [->, very thick] (7,5) -- ++(-1,0) node[left, thick] {\footnotesize $q_{-1,0}$};
\draw [->, very thick] (7,5) -- ++(-1,1) node[above left, thick] {\footnotesize $q_{-1,1}$};
\draw [->, very thick] (7,5) -- ++(0,-1) node[below, thick] {\footnotesize $q_{0,-1}$};
\draw [->, very thick] (7,5) -- ++(0,1) node[above, thick] {\footnotesize $q_{0,1}$}; 	
\draw [->, very thick] (7,5) -- ++(1,-1) node[below right, thick] {\footnotesize $q_{1,-1}$};
\draw [->, very thick] (7,5) -- ++(1,0) node[right, thick] {\footnotesize $q_{1,0}$};
\draw [->, very thick] (7,5) -- ++(1,1) node[above right, thick] {\footnotesize $q_{1,1}$}; 
	
\draw [->, very thick] (7,0) -- ++(0,1) node[above, thick] {\footnotesize $h_{0,1}$};
\draw [->, very thick] (7,0) -- ++(-1,1) node[above left, thick] {\footnotesize $h_{-1,1}$};
\draw [->, very thick] (7,0) -- ++(1,1) node[above right, thick] {\footnotesize $h_{1,1}$};
\draw [->, very thick] (7,0) -- ++(-1,0) node[below left, thick] {\footnotesize $h_{-1,0}$};
\draw [->, very thick] (7,0) -- ++(1,0) node[below right, thick] {\footnotesize $h_{1,0}$};
	
\draw [->, very thick] (0,5) -- ++(1,0) node[right, thick] {\footnotesize $v_{1,0}$};
\draw [->, very thick] (0,5) -- ++(1,-1) node[below right, thick] {\footnotesize $v_{1,-1}$};
\draw [->, very thick] (0,5) -- ++(1,1) node[above right, thick] {\footnotesize $v_{1,1}$};
\draw [->, very thick] (0,5) -- ++(0,1) node[above left, thick] {\footnotesize $v_{0,1}$};
\draw [->, very thick] (0,5) -- ++(0,-1) node[below left, thick] {\footnotesize $v_{0,-1}$};
	
\draw [->, very thick] (0,0) -- ++(1,0) node[below right, thick] {\footnotesize $r_{1,0}$};
\draw [->, very thick] (0,0) -- ++(0,1) node[above left, thick] {\footnotesize $r_{0,1}$};
\draw [->, very thick] (0,0) -- ++(1,1) node[above right, thick] {\footnotesize $r_{1,1}$};
\end{tikzpicture}
\caption{Components and possible transitions of $R$}
\label {fig:model_2d}
\end{figure}

Assume that $R$ is irreducible, aperiodic and positive recurrent. Then, let $\pi: S \rightarrow [0,1]$ be the stationary probability distribution of $R$, which means that $\pi$ satisfies that for any $(n_1,n_2) \in S$, 
\begin{eqnarray*}
\sum_{(i,j)\in N(n_1,n_2)} \pi(n_1+i,n_2+j)q_{-i,-j}(n_1+i,n_2+j) + \pi(n_1,n_2)q_{0,0}(n_1,n_2) = 0. 
\end{eqnarray*}
We remark that all the distributions mentioned throughout the paper are probability distributions. 

Consider a non-negative reward function on the state space, \ie $F: S \rightarrow [0,\infty)$. The main goal of the paper is to bound the following expected stationary performance, 
\begin{eqnarray}
\mathcal{F} = \sum_{(n_1,n_2)\in S} \pi(n_1,n_2) F(n_1,n_2). 
\end{eqnarray}

The approximation is obtained by considering the same reward function on a perturbed random walk $\bar{R}$, which is defined on the same partition of the state space, and the same set of possible transitions within each component. Denote by $\bar{\pi}$ the stationary probability distribution of $\bar{R}$. Then the expected stationary reward of $\bar{R}$ is given by 
\begin{eqnarray*}
\bar{\mathcal{F}} = \sum_{(n_1,n_2)\in S} \bar{\pi}(n_1,n_2) F(n_1,n_2).
\end{eqnarray*}
In this paper, $\bar{\pi}$ is considered to have the following form, 
\begin{eqnarray*}
\bar{\pi}(n_1,n_2) = \sum_{k=1}^{K} c_k \rho_k^{n_1}\sigma_k^{n_2}, \qquad \forall (n_1,n_2) \in S,
\end{eqnarray*}
with $(\rho_k,\sigma_k) \in (0,1)^2$ and $c_k \in \mathbb{R}$. This form will be referred to as a sum of geometric terms further in the remainder. In addition, $(\rho_k,\sigma_k)$ is chosen such that the single-term measure $\rho_k^{n_1}\sigma_k^{n_2}$ satisfies the interior balance equation. Define $\Gamma = \{ (\rho_1,\sigma_1),\dots,(\rho_K,\sigma_K) \}$. Moreover, let $\mathrm{Li}_s(z)$ denote the polylogarithm function, defined as
\begin{eqnarray*}
\mathrm{Li}_s(z) = \sum_{k=1}^\infty \frac{z^k}{k^s}. 
\end{eqnarray*}
In this paper, the following two problems will be discussed. 
\begin{enumerate}
\item For the given $\bar{\pi}$ that is a sum of geometric terms, inhomogeneous transition rates on the axes are constructed such that $\bar{\pi}$ is indeed the stationary probability of $\bar{R}$. It is shown that the rates on the axes are uniformly bounded for all states. Moreover, the rates are convergent as the boundary state goes to infinity. 
\item An explicit expression for the bound on $| \bar{\mathcal{F}} - \mathcal{F} |$ will be given. 
\end{enumerate}
These two problems will be tackled separately in the next two sections.

\section{Inhomogeneous perturbed random walk}
\label{sec:inhomo_perturbed}

\subsection{Construction of the inhomogeneous transition rates}

In this section, the transition rates of $\bar{R}$ will be constructed step by step. Denote by $\bar{q}_{i,j}(n_1,n_2)$, $\bar{h}_{i,j}(n_1,0)$, $\bar{v}(0,n_2)$ and $\bar{r}_{i,j}$ the transition rates in the interior, on the horizontal axis, the vertical axis and at the origin respectively. First, let the interior transition rates of $\bar{R}$ be equal to those of R. Second, on the horizontal and vertical axes, fix the rates of the transitions going from the axes to the interior. Moreover, for the transition rates along the axes, one of them is chosen to be constant for all the states, and the other depends on the state. The discussion above is reflected in the following theorem. 

\begin{theorem}
\label{thm:perturb_transition_rates}
Suppose that $\bar{\pi}$ is a sum of geometric terms satisfying the interior balance equation. Let $\bar{q}_{i,j}(n_1,n_2) = q_{i,j}, \forall (n_1,n_2) \in S_4, (i,j) \in N_4$. Moreover, on the horizontal and vertical axes let 
\begin{eqnarray*}
\bar{h}_{i,1}(n_1,0) = q_{i,1}, \quad \bar{v}_{i,j}(0,n_2) = q_{1,j}, \quad \forall i,j=-1,0,1,\ n_1,n_2 = 1,2,\dots. 
\end{eqnarray*}
Choose constants $\bar{h}_{1,0}, \bar{v}_{0,1} \ge 0$ and fix $\bar{r}_{1,0} = \bar{h}_{1,0}$, $\bar{r}_{0,1} = \bar{v}_{0,1}$, $\bar{r}_{1,1} = q_{1,1}$ at the origin and let
\begin{eqnarray*}
\bar{h}_{1,0}(n_1,0) = \bar{h}_{1,0}, \qquad \bar{v}_{0,1}(0,n_2) = \bar{v}_{0,1}, \qquad \forall n_1,n_2 = 1,2,\dots. 
\end{eqnarray*}
In addition, the homogeneous rates are 
\begin{align*}
\bar{h}_{-1,0}(n_1,0) &= \frac{\sum_{k=1}^K c_k\rho_k^{n_1-1}}{\sum_{k=1}^K c_k\rho_k^{n_1}}\bar{h}_{1,0} - \frac{\sum_{k=1}^K c_k\alpha_k(1-\rho_k)^{-1}\rho_k^{n_1}}{\sum_{k=1}^K c_k\rho_k^{n_1}}, \\
\bar{v}_{0,-1}(0,n_2) &= \frac{\sum_{k=1}^K c_k\sigma_k^{n_2-1}}{\sum_{k=1}^K c_k\sigma_k^{n_1}}\bar{v}_{0,1} - \frac{\sum_{k=1}^K c_k\beta_k(1-\sigma_k)^{-1}\sigma_k^{n_2}}{\sum_{k=1}^K c_k\sigma_k^{n_2}},
\end{align*}
for $n_1,n_2=1,2,\dots$, where
\begin{align*}
\alpha_k = \sum_{i=-1}^1 q_{i,1} - \sum_{i=-1}^1 \rho_k^{-i}\sigma_k q_{i,-1}, \qquad \beta_k = \sum_{j=-1}^1 q_{1,j} - \sum_{j=-1}^1 \rho_k \sigma_k^{-j} q_{-1,j}.
\end{align*}
Then the given $\bar{\pi}$ is the stationary probability distribution of $\bar{R}$. 
\end{theorem}
\begin{proof}
To show that $\bar{\pi}$ is the stationary probability distribution, it suffices to verify all the balance equations. As is specified, the transition rates in the interior are equal to those of $R$. Moreover, the transition rates from axes to interior are equal to those rates of R as well. Hence, the balance equation in the interior can be easily verified. Then, the balance equation on the horizontal axis is specified below, 
\begin{multline*}
\sum_{i=-1}^1 \bar{\pi}(n_1-i,1) q_{i,-1} + \bar{\pi}(n_1-1,0) \bar{h}_{1,0}(n_1-1,0) + \bar{\pi}(n_1+1,0) \bar{h}_{-1,0}(n_1+1,0) \\
= \bar{\pi}(n_1,0) \left( \sum_{i=-1}^1 \bar{h}_{i,1}(n_1,0) + \bar{h}_{-1,0}(n_1,0) + \bar{h}_{1,0}(n_1,0) \right), 
\end{multline*}
In the following part, we only check the balance equation on the horizontal axis, since the vertical equation can be checked in the same way. Rewriting the balance equation gives that, 
{\footnotesize
\begin{multline*}
\bar{\pi}(n_1-1,0) \bar{h}_{1,0}(n_1-1,0) + \bar{\pi}(n_1+1,0) \bar{h}_{-1,0}(n_1+1,0)
-\ \bar{\pi}(n_1,0) \left( \bar{h}_{-1,0}(n_1,0) + \bar{h}_{1,0}(n_1,0) \right) \\
= \bar{\pi}(n_1,0) \sum_{i=-1}^1 \bar{h}_{i,1}(n_1,0) - \sum_{i=-1}^1 \bar{\pi}(n_1-i,1) q_{i,-1}
\end{multline*}}
Next, we only need to plug in the transition rates given in the theorem and check if the equation above holds. Starting from the right hand side of the equation above, we have
\begin{eqnarray*}
RHS = \bar{\pi}(n_1,0) \sum_{i=-1}^1 q_{i,1} - \sum_{i=-1}^1 \bar{\pi}(n_1-i,1) q_{i,-1} = \sum_{k=1}^{K} c_k \alpha_k \rho_k^{n_1},
\end{eqnarray*}
where the last equality holds by definition of $\alpha_k$. Moreover, 
{\footnotesize
\begin{align*}
LHS &= \bar{\pi}(n_1-1,0)\bar{h}_{1,0} + \bar{\pi}(n_1+1,0) \bar{h}_{-1,0}(n_1+1,0) - \bar{\pi}(n_1,0)\bar{h}_{1,0} - \bar{\pi}(n_1,0)\bar{h}_{-1,0}(n_1,0) \\
&= \sum_{k=1}^{K} c_k \alpha_k (1-\rho_k)^{-1}\rho_k^{n_1} - \sum_{k=1}^K c_k \alpha_k (1-\rho_k)^{-1} \rho_k^{n_1+1} \\
&= \sum_{k=1}^{K} c_k \alpha_k \rho_k^{n_1} = RHS. 
\end{align*}
}
Hence, the balance equation on the horizontal axis holds. In the same fashion, balance equations on the vertical axis and at the origin can be verified. Therefore, we conclude that $\bar{\pi}$ is indeed the stationary probability distribution of $\bar{R}$ with the transition rates specified in the theorem. 
\end{proof}

In Theorem~\ref{thm:perturb_transition_rates}, one way to construct the transition rates of $\bar{R}$ is proposed. Further in the paper, this construction approach will be referred to as \emph{the single direction inhomogeneous construction}. 

In the following part, the behavior of $\bar{h}_{-1,0}(n_1,0)$ will be discussed. We first consider its convergence.

\subsection{Convergence of the inhomogeneous rates}

Define
\begin{eqnarray*}
\rho_* = \max\left\{ \rho_1,\dots,\rho_K \right\}, \qquad \sigma_{**} = \max\left\{ \sigma_1,\dots,\sigma_K \right\}. 
\end{eqnarray*}
Moreover, let $\sigma_*$($\rho_{**}$) be the weighted average of all the $\sigma$($\rho$)'s that pair up with $\rho_*$($\sigma_{**}$) in $\Gamma$, \ie
\begin{eqnarray*}
\sigma_* = \frac{\sum_{k:(\rho_*,\sigma_k) \in \Gamma} c_k\sigma_k}{\sum_{k:(\rho_*,\sigma_k) \in \Gamma} c_k}, \qquad \rho_{**} = \frac{\sum_{k:(\rho_k,\sigma_{**}) \in \Gamma} c_k\rho_k}{\sum_{k:(\rho_k,\sigma_{**}) \in \Gamma} c_k}. 
\end{eqnarray*}
There are at most two terms in the weighted average, according to the discussion in~\cite{chen2015invariant}. In the next theorem, it is shown that $\bar{h}_{-1,0}(n_1,0)$ and $\bar{v}_{0,-1}(0,n_2)$ converge as $n_1, n_2$ go to infinity and a close-form expression for the limiting rate is given. 

\begin{theorem}
\label{thm:limiting_probability}
Let $\bar{\pi}$ be a sum of geometric terms satisfying interior balance equation. Suppose that the transition rates of $\bar{R}$ are obtained by the single direction inhomogeneous construction. Then, $\bar{h}_{-1,0}(n_1,0)$ and $\bar{v}_{0,-1}(0,n_2)$ converge and the limits are, 
\begin{align*}
\lim_{n_1\rightarrow \infty} \bar{h}_{-1,0}(n_1,0) &= \rho_*^{-1}\bar{h}_{1,0} - \sum_{i=-1}^{1} (1-\rho_*)^{-1} \left( q_{i,1} - \rho_*^{-i}\sigma_*q_{i,-1} \right), \\
\lim_{n_2\rightarrow \infty} \bar{v}_{0,-1}(0,n_2) &= \sigma_{**}^{-1}\bar{v}_{0,1} - \sum_{j=-1}^{1} (1-\sigma_{**})^{-1} \left( q_{1,j} - \rho_{**}\sigma_{**}^{-j}q_{-1,j} \right). 
\end{align*}
\end{theorem}
\begin{proof}
We only give the proof for the convergence of $\bar{h}_{-1,0}(n_1,0)$. The proof for $\bar{v}_{0,-1}(0,n_2)$ follows in similar fashion. From Theorem~\ref{thm:perturb_transition_rates}, it is seen that
\begin{eqnarray*}
\bar{h}_{-1,0}(n_1,0) &=& \frac{\sum_{k=1}^K c_k\rho_k^{n_1-1}}{\sum_{k=1}^K c_k\rho_k^{n_1}}\bar{h}_{1,0} - \frac{\sum_{k=1}^K c_k\alpha_k\rho_k^{n_1} (1-\rho_k)^{-1}}{\sum_{k=1}^K c_k\rho_k^{n_1}}. 
\end{eqnarray*}
Consider the quotient of sums of geometric terms, it holds that
\begin{align*}
\frac{\sum_{k=1}^K c_k\rho_k^{n_1-1}}{\sum_{k=1}^K c_k\rho_k^{n_1}} &= \frac{\rho_*^{n_1-1} [ \sum_{k:\rho_k = \rho_*} c_k + \sum_{k:\rho_k \neq \rho_*} c_k(\frac{\rho_k}{\rho_*})^{n_1-1} ]}{\rho_*^{n_1} [ \sum_{k:\rho_k = \rho_*} c_k + \sum_{k:\rho_k \neq \rho_*} c_k(\frac{\rho_k}{\rho_*})^{n_1} ]} \\
&= \rho_*^{-1} \cdot \frac{\sum_{k:\rho_k = \rho_*}c_k + \sum_{k:\rho_k \neq \rho_*} c_k(\frac{\rho_k}{\rho_*})^{n_1-1}}{\sum_{k:\rho_k = \rho_*} c_k + \sum_{k:\rho_k \neq \rho_*} c_k(\frac{\rho_k}{\rho_*})^{n_1}}.
\end{align*}
Since $\rho_k/\rho_* < 1$ for $\rho_{k} \neq \rho_*$, by letting $n_1$ go to infinity on both sides, we obtain
\begin{eqnarray*}
\lim_{n_1\rightarrow\infty} \frac{\sum_{k=1}^{K} c_k\rho_k^{n_1-1}}{\sum_{k=1}^{K} c_k\rho_k^{n_1}}\bar{h}_{1,0} = \rho_*^{-1}\bar{h}_{1,0}.
\end{eqnarray*}
Recall that $\alpha_k = \sum_{i=-1}^1 q_{i,1} - \sum_{i=-1}^1 \rho_k^{-i}\sigma_k q_{i,-1}$. Then the following limits hold, 
\begin{align*}
\lim_{n_1\rightarrow \infty} \frac{\sum_{k=1}^{K} c_k (1-\rho_k)^{-1}\rho_k^{n_1} q_{i,1}}{\sum_{k=1}^{K} c_k\rho_k^{n_1}} &= (1-\rho_*)^{-1} q_{i,1}, \qquad \forall i = -1,0,1, \\
\lim_{n_1\rightarrow \infty} \frac{\sum_{k=1}^{K} c_k (1-\rho_k)^{-1}\rho_k^{n_1-i}\sigma_k q_{i,-1}}{\sum_{k=1}^{K} c_k\rho_k^{n_1}} &= (1-\rho_*)^{-1}\rho_*^{-i}\sigma_* q_{i,-1}, \qquad \forall i = -1,0,1. 
\end{align*}
Therefore, the limiting rate is given by
\begin{eqnarray*}
\lim_{n_1\rightarrow \infty} \bar{h}_{-1,0}(n_1,0) = \rho_*^{-1}\bar{h}_{1,0} - \sum_{i=-1}^{1} (1-\rho_*)^{-1} \left( q_{i,1} - \rho_*^{-i}\sigma_* q_{i,-1} \right). 
\end{eqnarray*}
\end{proof}

From the convergence result, the inhomogeneous rates are bounded as well. Hence, the following corollary follows immediately. 
\begin{corollary}
\label{cor:uniformizable}
Let $\bar{\pi}$ be a sum of geometric terms satisfying interior balance equation. Suppose that the transition rates of $\bar{R}$ are obtained by the single direction inhomogeneous construction. Then there exists a constant $0 < \bar{\gamma} < \infty$, such that for any $(n_1,n_2)\in S$,
\begin{eqnarray*}
\sum_{(i,j)\in N(n_1,n_2)} \bar{q}_{i,j}(n_1,n_2) \le \bar{\gamma}. 
\end{eqnarray*}
\end{corollary}

In addition, it is observed that for any chosen $\bar{h}_{1,0}$, the limiting rates depend only on $(\rho_*,\sigma_*)$ or $(\rho_{**},\sigma_{**})$ out of the parameters in $\bar{\pi}$. Hence, it is clear that if $\bar{h}_{1,0}, \bar{v}_{0,1}$ are chosen such that the following equations hold,
\begin{align*}
\rho_*^{-1}\bar{h}_{1,0} - \sum_{i=-1}^{1} (1-\rho_*)^{-1} \left( q_{i,1} - \rho_*^{-i}\sigma_*q_{i,-1} \right) = h_{-1,0}, \\
\sigma_{**}^{-1}\bar{v}_{0,1} - \sum_{j=-1}^{1} (1-\sigma_{**})^{-1} \left( q_{1,j} - \rho_{**}\sigma_{**}^{-j}q_{-1,j} \right) = v_{0,-1}, 
\end{align*}
the limiting rates will be the same as $h_{-1,0}$ and $v_{0,-1}$. This choice will be reconsidered for the error bound in Section~\ref{sec:numerical_experiment}. In this subsection we have discussed the limit of the inhomogeneous transition rates. In the next subsection, we consider the condition under which the inhomogeneous rates are above the original rate for all states.

\subsection{Sufficient conditions for $\bar{h}_{-1,0}(n_1,0) \ge h_{-1,0}$, $\bar{v}_{0,-1}(0,n_2) \ge v_{0,-1}$} 

From Corollary~\ref{cor:uniformizable}, it is seen that $\bar{R}$ is uniformizable. Moreover, it can be observed that for any $n_1$, $\bar{h}_{-1,0}(n_1,0)$ is a linear function in $\bar{h}_{1,0}$ with a positive slope. Then, the larger $\bar{h}_{1,0}$ is chosen to be, the larger $\bar{h}_{-1,0}(n_1,0)$ gets in general. Hence, due to the convergence of the inhomogeneous transition rates, there are thresholds for $\bar{h}_{1,0}$ and $\bar{v}_{0,1}$ such that $\bar{h}_{-1,0}(n_1,0) \ge h_{-1,0}$ and $\bar{v}_{0,-1}(0,n_2) \ge v_{0,-1}$. The thresholds will be used for the error bound result in the next section.

However, it is unclear how the thresholds depend on the given parameters. This is because that properties such as monotonicity or zeros of the sum of geometric terms with negative coefficients are still open problems. In the next theorem, it is seen that if all the coefficients are positive, the thresholds for $\bar{h}_{1,0}$ and $\bar{v}_{0,1}$ can be related to the parameters $\rho$'s and $\sigma$'s. 

\begin{theorem}
\label{thm:inhomo_transition_above_threshold_postivie_coefficient}
Let $\bar{\pi}$ be a sum of geometric terms satisfying interior balance equation. Suppose that the transition rates of $\bar{R}$ are obtained by the single direction inhomogeneous construction. Moreover, assume that $c_k > 0$ for all $k=1,\dots,K$. Then, choosing
\begin{align*}
\bar{h}_{1,0} &\ge \max_{k=1,\dots,K} \left\{ \rho_k h_{-1,0}+\rho_k\alpha_k(1-\rho_k)^{-1} \right\}, \\
\bar{v}_{0,1} &\ge \max_{k=1,\dots,K} \left\{ \sigma_k v_{0,-1}+\sigma_k\beta_k(1-\sigma_k)^{-1} \right\}.  
\end{align*}
gives that $\bar{h}_{-1,0}(n_1,0) \ge h_{-1,0}$, $\bar{v}_{0,-1}(0,n_2) \ge v_{0,-1}$ for $n_1,n_2=1,2,\dots$. 
\end{theorem}
\begin{proof}
From Theorem~\ref{thm:perturb_transition_rates}, $\bar{h}_{-1,0}(n_1,0) \ge h_{-1,0}$ if and only if for any $n_1=1,2,\dots$,
\begin{eqnarray*}
\frac{\sum_{k=1}^K c_k\rho_k^{n_1-1}}{\sum_{k=1}^K c_k\rho_k^{n_1}}\bar{h}_{1,0} - \frac{\sum_{k=1}^K c_k\alpha_k(1-\rho_k)^{-1}\rho_k^{n_1}}{\sum_{k=1}^K c_k\rho_k^{n_1}} \ge h_{-1,0}, 
\end{eqnarray*}
which is equivalent to
\begin{eqnarray*}
\bar{h}_{1,0} \ge \frac{\sum_{k=1}^K c_k\rho_k^{n_1}}{\sum_{k=1}^K c_k\rho_k^{n_1-1}} h_{-1,0} + \frac{\sum_{k=1}^K c_k\alpha_k(1-\rho_k)^{-1} \rho_k^{n_1}}{\sum_{k=1}^K c_k\rho_k^{n_1-1}}.
\end{eqnarray*}
Rewriting the right hand side gives that 
\begin{eqnarray*}
\bar{h}_{1,0} &\ge& \frac{\sum_{k=1}^K \left( c_k\rho_k^{n_1}h_{-1,0} + c_k\alpha_k(1-\rho_k)^{-1} \rho_k^{n_1} \right)}{\sum_{k=1}^K c_k\rho_k^{n_1-1}} \\
&=& \frac{\sum_{k=1}^K c_k\rho_k^{n_1-1} \left( \rho_k h_{-1,0} + \rho_k\alpha_k(1-\rho_k)^{-1} \right)}{\sum_{k=1}^K c_k\rho_k^{n_1-1}}
\end{eqnarray*}
Observing the right hand side, it is concluded that for any $n_1$, it is a weighted average of $\rho_k h_{-1,0} + \rho_k\alpha_k(1-\rho_k)^{-1}$ over $k$. Therefore, for any $n_1=1,2,\dots$, 
{\small
\begin{eqnarray*}
\frac{\sum_{k=1}^K c_k\rho_k^{n_1-1} \left( \rho_k h_{-1,0} + \rho_k\alpha_k(1-\rho_k)^{-1} \right)}{\sum_{k=1}^K c_k\rho_k^{n_1-1}} \le \max_{k=1,\dots,K} \left\{ \rho_k h_{-1,0} + \rho_k\alpha_k(1-\rho_k)^{-1} \right\}.
\end{eqnarray*}
}
For the vertical axis, the same argument can be used for the proof. 
\end{proof}

Summing up the results in this section, a way to construct inhomogeneous transition rates is proposed for the given sum of geometric terms $\bar{\pi}$. Moreover, the properties of the constructed transition rates are discussed. In the next section, the error bound will be considered and an explicit expression will be given.

\section{Error bounds on the average stationary performance}
\label{sec:error_bound}

\subsection{Markov reward approach}

In this section, the Markov reward approach is used to derive the error bound, which has been developed by van Dijk. The first step of applying the Markov reward approach is to obtain equivalent discrete-time Markov chains corresponding to $R$ and $\bar{R}$ using the standard uniformization approach. This is possible since both $R$ and $\bar{R}$ are uniformizable, \ie there exists $\gamma < \infty$, for which 
\begin{eqnarray*}
\sum_{(i,j) \in N(n_1,n_2)} q_{i,j}(n_1,n_2) \le \gamma, \quad \sum_{(i,j) \in N(n_1,n_2)} \bar{q}_{i,j}(n_1,n_2) \le \gamma, \qquad \forall (n_1,n_2)\in S. 
\end{eqnarray*}
Then the uniformized discrete-time Markov chains have the following transition probabilities,
\begin{align*}
p_{i,j}(n_1,n_2) &= \left\{
\begin{array}{l@{\qquad}l}
\gamma^{-1}q_{i,j}(n_1,n_2), & (i,j) \neq (0,0), \\
1 - \gamma^{-1} \sum_{(i,j) \neq (0,0)} q_{i,j}(n_1,n_2), & (i,j) = (0,0), 
\end{array}
\right. \\
\bar{p}_{i,j}(n_1,n_2) &= \left\{
\begin{array}{l@{\qquad}l}
\gamma^{-1}\bar{q}_{i,j}(n_1,n_2), & (i,j) \neq (0,0), \\
1 - \gamma^{-1} \sum_{(i,j) \neq (0,0)} \bar{q}_{i,j}(n_1,n_2), & (i,j) = (0,0). 
\end{array}
\right. 
\end{align*}
The constant $\gamma$ is used in the uniformization step. It is not used in the main result of the Markov reward approach. 

The reward function $F(n)$ is seen as a one-step reward when the discrete-time Markov chains stays in $n$ for one time unit. Define $F^t(n_1,n_2)$ as the expected cumulative reward up to time $t$ starting from state $(n_1,n_2)$ at time $0$, \ie
{\small
\begin{eqnarray*}
F^t(n_1,n_2) = \left\{
\begin{array}{ll}
0, & t = 0, \\
F(n_1,n_2) + \sum_{(i,j) \in N(n_1,n_2)} p_{i,j}(n_1,n_2) F^{t-1}(n_1+i,n_2+j), & t \geq 1. 
\end{array}\right. 
\end{eqnarray*}
}
Then, under the ergodicity condition, it holds that 
\begin{eqnarray*}
\lim_{t\rightarrow \infty} \frac{F^t(n_1,n_2)}{t} = \mathcal{F}, \qquad \forall (n_1,n_2) \in S. 
\end{eqnarray*}
Moreover, for $(n_1,n_2) \in S$, define the bias terms as follows, 
\begin{eqnarray*}
D^t_{i,j}(n_1,n_2) = F^t(n_1+i,n_2+j) - F^t(n_1,n_2),
\end{eqnarray*}
The main result for building up the error bound is given below. Here the result is presented in the convenience of our model and notation. A more general form of the result is given in~\cite{vandijk11inbook}. 

\begin{theorem}[Result 9.3.2 in~\cite{vandijk11inbook}]
\label{thm:markov_reward_result}
Suppose that there exist $G_1: S_1 \rightarrow [0,\infty)$, $G_2: S_2 \rightarrow [0,\infty)$ and $G_3 \ge 0$ for which
\begin{align}
\label{eq:construct_error_horizontal}
\left| \sum_{(i,j) \in N_1} \left[ \bar{h}_{i,j}(n_1,0) - h_{i,j} \right] D^t_{i,j}(n_1,0) \right| &\le G_1(n_1,0), \quad \forall n_1=1,2,\dots, \\
\label{eq:construct_error_vertical}
\left| \sum_{(i,j) \in N_2} \left[ \bar{v}_{i,j}(0,n_2) - v_{i,j} \right] D^t_{i,j}(0,n_2) \right| &\le G_2(0,n_2), \quad \forall n_2=1,2,\dots, \\
\label{eq:construct_error_origin}
\left| \sum_{(i,j) \in N_3} \left[ \bar{r}_{i,j} - r_{i,j} \right] D^t_{i,j}(0,0) \right| &\le G_3, 
\end{align}
for all $t \ge 0$. Then 
\begin{eqnarray*}
\left| \bar{\mathcal{F}} - \mathcal{F} \right| \le \sum_{n_1=1}^\infty \bar{\pi}(n_1,0) G_1(n_1,0) + \sum_{n_2=1}^\infty \bar{\pi}(0,n_2) G_2(0,n_2) + \bar{\pi}(0,0) G_3. 
\end{eqnarray*}
\end{theorem} 

Notice that conditions in Equation~\eqref{eq:construct_error_horizontal}~-~\eqref{eq:construct_error_origin} need to hold for all $t\ge 0$, hence it is important to obtain bounds on $D^t_{i,j}(n_1,n_2)$ that are independent of $t$. Indeed, in~\cite{vandijk11inbook} a lot of efforts have been spent on finding those bounds. However, this is not our focus in this paper. We would like to derive the error bound explicitly for the proposed inhomogeneous perturbation framework. Hence it is assumed that there exist functions $B_1(n_1,0)$, $B_2(0,n_2)$ and $B_3$ for which
\begin{align*}
\left| D^t_{i,j}(n_1,0) \right| \le B_1(n_1,0), & \forall (i,j) \in N_1,\\
\left| D^t_{i,j}(0,n_2) \right| \le B_2(0,n_2), & \forall (i,j) \in N_2, \\
\label{eq:bounds_bias_terms}
\left| D^t_{i,j}(0,0) \right| \le B_3, & \forall (i,j) \in N_3, \tag{$\star$}
\end{align*}
and that $B_1(n_1,0)$, $B_2(0,n_2)$ are of the following form,
\begin{align*}
B_1(n_1,0) = \sum_{m=0}^M b_{1,m} n_1^{m}, \qquad B_1(0,n_2) = \sum_{m=0}^M b_{2,m} n_2^{m}.
\end{align*}
So far, the preliminaries for deriving the error bound have been presented. In the next subsection, we will obtain an explicit expression for the error bound.

\subsection{Derivation of the error bound}

Consider a state $(n_1,0)$ on the horizontal axis, by~\eqref{eq:bounds_bias_terms} it holds that
{\small
\begin{eqnarray*}
\left| \sum_{(i,j) \in N_1} \left[ \bar{h}_{i,j}(n_1,0) - h_{i,j} \right] D^t_{i,j}(n_1,0) \right| \le \sum_{(i,j) \in N_1} \left| \bar{h}_{i,j}(n_1,0) - h_{i,j} \right| B_1(n_1,0).
\end{eqnarray*}
}
Hence, $G_1(n_1,0)$ will be taken to be
\begin{eqnarray*}
G_1(n_1,0) = \sum_{\footnotesize (i,j) \in N_1} \left| \bar{h}_{i,j}(n_1,0) - h_{i,j} \right| B_1(n_1,0). 
\end{eqnarray*}
Since the transition rates are obtained by the single direction inhomogeneous construction, only $h_{-1,0}(n_1,0)$ depends on $n_1$. Denote by
\begin{eqnarray*}
\delta_h = \sum_{i=-1}^1 \left| \bar{h}_{i,1} - h_{i,1} \right|, \qquad \delta_v = \sum_{j=-1}^1 \left| \bar{v}_{1,j} - v_{1,j} \right|, \qquad \delta_r = \sum_{(i,j)\in N_3} \left| \bar{r}_{1,j} - r_{i,j} \right|. 
\end{eqnarray*}
Then the sum above can be split as
\begin{eqnarray}
\label{eq:g_function_horizontal}
G_1(n_1,0) = B_1(n_1,0) \left( \delta_h + \left| \bar{h}_{1,0} - h_{1,0} \right| + \left| \bar{h}_{-1,0}(n_1,0) - h_{-1,0} \right| \right).
\end{eqnarray}
In Equation~\eqref{eq:g_function_horizontal}, the absolute value $| \bar{h}_{-1,0}(n_1,0) - h_{-1,0} |$ makes it more difficult to calculate the summation $\sum_{n_1=1}^\infty \bar{\pi}(n_1,0) G_1(n_1,0)$. Recall from Theorem~\ref{thm:inhomo_transition_above_threshold_postivie_coefficient} that it is possible to choose $\bar{h}_{1,0}$ such that for any $n_1=1,2,\dots$, $\bar{h}_{-1,0}(n_1,0) \ge h_{-1,0}$. This will get rid of the absolute value and make error bound calculation feasible. Therefore, the following theorem gives an explicit expression for the error bound. 

\begin{theorem}
\label{thm:error_bound_result}
Let $R$ be a continuous-time random walk and $\bar{\pi}$ be a sum of geometric terms satisfying the interior balance equation. Suppose that the transition rates of $\bar{R}$ are obtained by the single direction inhomogeneous construction. Moreover, suppose that $\bar{h}_{1,0}$ and $\bar{v}_{0,1}$ are chosen such that 
\begin{eqnarray*}
\bar{h}_{-1,0}(n_1,0) \ge h_{-1,0}, \qquad \bar{v}_{0,-1}(0,n_2) \ge v_{0,-1}, \qquad \forall n_1,n_2 = 1,2,\dots. 
\end{eqnarray*}
Then,  
\begin{eqnarray}
\label{eq:error_bound_result}
\left| \bar{\mathcal{F}} - \mathcal{F} \right| \leq g^{(1)} + g^{(2)} + g^{(3)},
\end{eqnarray}
where
{\footnotesize
\begin{align*}
g^{(1)} =& \left( \delta_h + \left| \bar{h}_{1,0} - h_{1,0} \right| - h_{-1,0} \right)\sum_{k=1}^K \sum_{m=0}^M c_k b_{1,m} \mathrm{Li}_{-m}(\rho_k) + \bar{h}_{1,0}\sum_{k=1}^K \sum_{m=0}^M c_k\rho_k^{-1} b_{1,m} \mathrm{Li}_{-m}(\rho_k) \\
& - \sum_{k=1}^K\sum_{m=0}^M c_k\alpha_k b_{1,m}(1-\rho_k)^{-1}\mathrm{Li}_{-m}(\rho_k), \\
g^{(2)} =& \left( \delta_v + \left| \bar{v}_{0,1} - v_{0,1} \right|- v_{0,-1} \right)\sum_{k=1}^K \sum_{m=0}^M c_k b_{2,m} \mathrm{Li}_{-m}(\sigma_k) + 
\bar{v}_{0,1}\sum_{k=1}^K \sum_{m=0}^M c_k\sigma_k^{-1} b_{2,m} \mathrm{Li}_{-m}(\sigma_k) \\
& - \sum_{k=1}^K\sum_{m=0}^M c_k\beta_k b_{2,m}(1-\sigma_k)^{-1}\mathrm{Li}_{-m}(\sigma_k), \\
g^{(3)} =& \delta_r B_3\sum_{k=1}^K c_k.
\end{align*}
}
\end{theorem}
\begin{proof}
From Theorem~\ref{thm:markov_reward_result}, it is seen that a calculation of the following,
\begin{eqnarray*}
g^{(1)} = \sum_{n_1=1}^\infty \bar{\pi}(n_1,0)G_1(n_1,0), \quad g^{(2)} = \sum_{n_2=1}^\infty \bar{\pi}(0,n_2)G_2(0,n_2), \quad g^{(3)} = \bar{\pi}(0,0)G_3,
\end{eqnarray*}
is sufficient for the error bound. In the following proof, only the calculation for $g^{(1)}$ is given and the rest follows the same fashion.  

Since the choice of $\bar{h}_{1,0}$ ensures that $\bar{h}_{-1,0}(n_1,0) \ge h_{-1,0}$ for all $n_1=1,2,\dots$, using Equation~\eqref{eq:g_function_horizontal} we get that 
{\footnotesize
\begin{align*}
g^{(1)} &= \sum_{n_1 = 1}^{\infty} \bar{\pi}(n_1,0) \cdot \left[ \left( \delta_h + \left| \bar{h}_{1,0} - h_{1,0} \right| + \bar{h}_{-1,0}(n_1,0) - h_{-1,0} \right)\cdot B(n_1,0) \right], \\
&= \left( \delta_h + \left| \bar{h}_{1,0} - h_{1,0} \right| - h_{-1,0} \right) \sum_{n_1=1}^\infty \bar{\pi}(n_1,0)B_1(n_1,0) + \sum_{n_1=1}^\infty \bar{h}_{-1,0}(n_1,0)\bar{\pi}(n_1,0)B_1(n_1,0),
\end{align*}
}
where in the last step a split between constant terms and inhomogeneous term is used again. For the first summation, a direct calculation gives that 
\begin{eqnarray*}
\sum_{n_1=1}^\infty \bar{\pi}(n_1,0)B_1(n_1,0) =  \sum_{k=1}^K \sum_{m=0}^M\sum_{n_1=1}^\infty c_k b_{1,m}n_1^{m}\rho_k^{n_1} = \sum_{k=1}^K \sum_{m=0}^M c_k b_{1,m} \mathrm{Li}_{-m}(\rho_k). 
\end{eqnarray*}
For the second summation, from the construction of $\bar{h}_{-1,0}(n_1,0)$ we have
\begin{eqnarray*}
\bar{\pi}(n_1,0)\bar{h}_{-1,0}(n_1,0) - \bar{\pi}(n_1-1,0)\bar{h}_{1,0} = -\sum_{k=1}^K c_k \alpha_k(1-\rho_k)^{-1} \rho_k^{n_1}.
\end{eqnarray*}
Multiplying by $B(n_1,0)$ on both sides and summing up over $n_1$, we have
\begin{align*}
& \sum_{n_1=1}^{\infty} \bar{\pi}(n_1,0)\bar{h}_{-1,0}(n_1,0)B_1(n_1,0) \\ 
=& \sum_{n_1=1}^{\infty}\sum_{k=1}^K \bar{h}_{1,0} c_k\rho_k^{n_1-1} \sum_{m=0}^M b_{1,m} n_1^m - \sum_{n_1=1}^\infty \sum_{k=1}^K \sum_{m=0}^M c_k\alpha_k b_{1,m}n_1^m \rho_k^{n_1}(1-\rho_k)^{-1}. \\
\end{align*}
Then, it is easy to verify that the conclusion holds.
\end{proof}

In particular, if the bias terms are bounded by constants, \ie $M=0$, we can obtain an expression without the polylogarithm function, which is shown in the following corollary. 

\begin{corollary}
\label{cor:error_bound_degree_0}
Suppose that all the conditions in Theorem~\ref{thm:error_bound_result} hold. In addition, assume that $\left| D^t_{i,j}(n_1,0) \right| \le \mathcal{B}_1$ for all $(i,j) \in N_1$, $\left| D^t_{i,j}(0,n_2) \right| \le \mathcal{B}_2$ for all $(i,j) \in N_2$, and $\left| D^t_{i,j}(0,0) \right| \le \mathcal{B}_3$ for all $(i,j) \in N_3$. Then,  
\begin{eqnarray}
\left| \bar{\mathcal{F}} - \mathcal{F} \right| \leq g^{(1)} + g^{(2)} + g^{(3)},
\end{eqnarray}
where
\begin{align*}
g^{(1)} =& \left( \delta_h + \left| \bar{h}_{1,0} - h_{1,0} \right| - h_{-1,0} \right)\mathcal{B}_1\sum_{k=1}^K c_k\rho_k(1-\rho_k)^{-1} + \\ 
& \bar{h}_{1,0}\mathcal{B}_1\sum_{k=1}^K c_k(1-\rho_k)^{-1} - \mathcal{B}_1\sum_{k=1}^K c_k\alpha_k\rho_k(1-\rho_k)^{-2}, \\
g^{(2)} =& \left( q_v + \left| \bar{v}_{0,1} - v_{0,1} \right|- v_{0,-1} \right)\mathcal{B}_2\sum_{k=1}^K c_k\sigma_k(1-\sigma_k)^{-1} + \\
& \bar{v}_{0,1}\mathcal{B}_2\sum_{k=1}^K c_k(1-\sigma_k)^{-1} - \mathcal{B}_2\sum_{k=1}^K c_k\beta_k\sigma_k(1-\sigma_k)^{-2}, \\
g^{(3)} =& q_r \mathcal{B}_3\sum_{k=1}^K c_k.
\end{align*}
\end{corollary}
This expression will be used in the next section for the numerical calculation of the error bound. 

In this section, it is seen that after obtaining the transition rates of $\bar{R}$ using the single direction inhomogeneous construction, an explicit expression for the error bound can be derived. In the next section, the inhomogeneous perturbation will be applied in a numerical example and the error bound result will be plotted.

\section{Numerical experiments: random walk with joint departures}
\label{sec:numerical_experiment}

Consider a random walk with joint departures in the quarter plane. This model often appears in communication networks and has been studied in~\cite{goseling2013energy}. Homogeneous perturbation for this model is applied in~\cite{goseling2016linear} and bounds on the stationary performance are obtained using a linear programming approach. Here we apply the single direction inhomogeneous perturbation on this model and derive the error bound, in comparison with the one obtained in~\cite{goseling2016linear}. 

The model describes two queues with independent Poisson arrivals and simultaneous departures from both queues. When one of the queues is empty, the other one serves at a lower rate. A general introduction on the model is given in~\cite{goseling2016linear}. In this section, we consider a symmetric scenario, which means that both queues serve at the same rate in case the other one is empty. Thus, the non-zero transition rates are, 
\begin{align*}
& q_{1,0} = \lambda, \qquad q_{0,1} = \lambda, \qquad q_{-1,-1} = \mu, \\
& h_{1,0} = \lambda, \qquad h_{0,1} = \lambda, \qquad h_{-1,0} = \mu^*, \\
& v_{1,0} = \lambda, \qquad v_{0,1} = \lambda, \qquad v_{0,-1} = \mu^*.
\end{align*}
Moreover, assume that $2\lambda + \mu = 1$, and $\lambda < \mu^* < \mu$. The transition structure of this random walk is given in Figure~\ref{fig:transition_rw_joint}. 

\begin{figure}[htb!]
\centering
\begin{tikzpicture}[scale = 0.7]
\draw [->, >=stealth'] (0,0) -- (0,8) node[above, thick] {$n_2$};
\draw [->, >=stealth'] (0,0) -- (10,0) node[right, thick] {$n_1$};

\draw [->, very thick] (7,5) -- ++(-1,-1) node[below left, thick] {\footnotesize $\mu$};
\draw [->, very thick] (7,5) -- ++(0,1) node[above, thick] {\footnotesize $\lambda$}; 	
\draw [->, very thick] (7,5) -- ++(1,0) node[right, thick] {\footnotesize $\lambda$};

\draw [->, very thick] (7,0) -- ++(0,1) node[above, thick] {\footnotesize $\lambda$};
\draw [->, very thick] (7,0) -- ++(1,0) node[below right, thick] {\footnotesize $\lambda$};
\draw [->, very thick] (7,0) -- ++(-1,0) node[below left, thick] {\footnotesize $\mu^*$};

\draw [->, very thick] (0,5) -- ++(1,0) node[right, thick] {\footnotesize $\lambda$};
\draw [->, very thick] (0,5) -- ++(0,1) node[above left, thick] {\footnotesize $\lambda$};
\draw [->, very thick] (0,5) -- ++(0,-1) node[below left, thick] {\footnotesize $\mu^*$};

\draw [->, very thick] (0,0) -- ++(1,0) node[below right, thick] {\footnotesize $\lambda$};
\draw [->, very thick] (0,0) -- ++(0,1) node[above left, thick] {\footnotesize $\lambda$};
\end{tikzpicture}
\caption{Random walk with joint departures}
\label{fig:transition_rw_joint}
\end{figure}
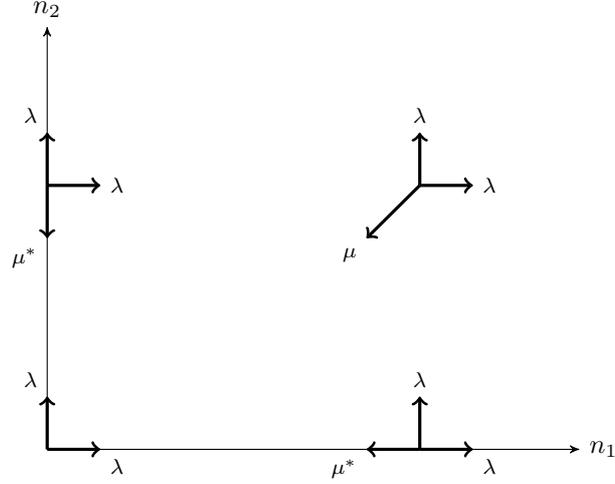

For the perturbed random walk $\bar{R}$, in Sub-section~\ref{ssec:rw_joint_homo} we describe the homogeneous perturbation considered in~\cite{goseling2016linear}. Then in Sub-section~\ref{ssec:rw_joint_inhomo}, the inhomogeneous perturbation scheme is given and necessary parameters are specified. 

\subsection{Homogeneous perturbation}
\label{ssec:rw_joint_homo}

In this case, consider homogeneous rates on the axes for $\bar{R}$ as well. The homogeneous perturbation is considered in~\cite{goseling2016linear} as an example and the given $\bar{\pi}$ has a product form, \ie
\begin{eqnarray*}
\bar{\pi}(n_1,n_2) = (1-\rho)(1-\sigma) \rho^{n_1}\sigma^{n_2}, \qquad (n_1,n_2) \in S. 
\end{eqnarray*}
More precisely, let
\begin{eqnarray*}
\rho = \sigma = \frac{-\mu + \sqrt{\mu^2+8\lambda\mu}}{2\mu}
\end{eqnarray*}
Then for the transition rates of $\bar{R}$, the service rates on the axes are
\begin{eqnarray*}
\bar{h}_{-1,0} = \mu/2, \qquad \bar{v}_{0,-1} = \mu/2,
\end{eqnarray*}
and the other rates are the same as those of $R$. 
In Figure~\ref{fig:geometric_term_rw_joint_homo}, all the curves for $R$ are plotted and $(\rho,\sigma)$ is marked with solid square. We consider that $\mu^* = 0.3\mu$. In addition, the curves of $\bar{R}$, $\bar{H}$ and $\bar{V}$, are also plotted so as to see how the perturbation modifies the curves. 

\begin{figure}[htb!]
\centering
\begin{tikzpicture}[scale=1]
\begin{axis}[
width=.7\textwidth,
xmax = 2, xmin = 0,
ymax = 2, ymin = 0,
font = \scriptsize,
legend style={
cells={anchor=west},
legend pos=north east,
font = \scriptsize,
}
]

\addplot[
line width=.3mm,color=red,
mark=none
]
table[
x index=0, y index=1, col sep = comma
]
{allcurves_data_int.csv};
\addlegendentry{$Q$};

\addplot+[
forget plot,
line width=.3mm,color=red,
mark=none
]
table[
x index=0, y index=2, col sep = comma
]
{allcurves_data_int.csv};
	
\addplot[
line width=.3mm,color=blue, 
mark=diamond, mark repeat = 10, mark phase = 5, mark size = 0.8mm
]
table[
x index=0, y index=1, col sep = comma
]
{allcurves_data_hor.csv};
\addlegendentry{$H$};

\addplot[
line width=.3mm,color=green,  
mark=o, mark repeat = 10, mark phase = 5, mark size = 0.8mm
]
table[
x index=1, y index=0, col sep = comma
]
{allcurves_data_ver.csv};
\addlegendentry{$V$};

\addplot[
line width=.3mm,color=blue, dashed, 
mark=diamond, mark repeat = 10, mark phase = 5, mark size = 0.8mm, mark options = {solid}
]
table[
x index=0, y index=1, col sep = comma
]
{allcurves_data_prod_hor.csv};
\addlegendentry{$\bar{H}$};

\addplot[
line width=.3mm,color=green, dashed,
mark=o, mark repeat = 10, mark phase = 5, mark size = 0.8mm, mark options = {solid}
]
table[
x index=1, y index=0, col sep = comma
]
{allcurves_data_prod_ver.csv};
\addlegendentry{$\bar{V}$};

\addplot[
mark=none
]
coordinates {(0,1) (1,1) (1,0)};

\addplot[
mark=square*, mark size = 3,
color=darkblue,
dashed
]
coordinates {(0.4574,0.4574)};

\end{axis}
\end{tikzpicture}
\caption{Characteristic curves of $R$ and $\bar{R}$: $\mu^*=0.3\mu$. }
\label{fig:geometric_term_rw_joint_homo}
\end{figure}
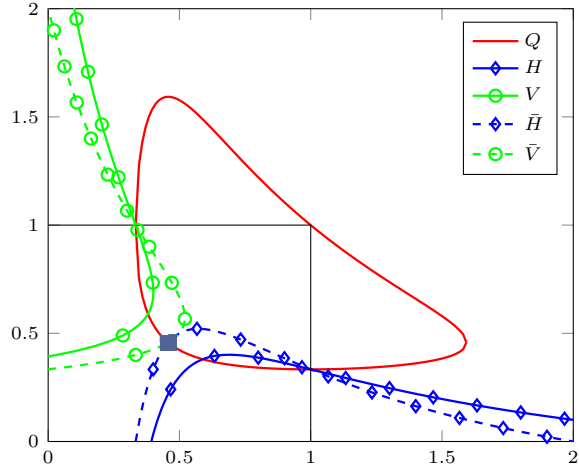

In Figure~\ref{fig:geometric_term_rw_joint_homo}, the curves of $R$ are in solid lines and those of $\bar{R}$ are plotted in dashed lines. Besides, the curve $H$ is marked with diamonds and $V$ is marked with circles. Indeed, it can be seen that $(\rho,\sigma)$ lies on the curves $Q$, $\bar{H}$ and $\bar{V}$ hence $\bar{\pi}$ is the stationary probability distribution of $\bar{R}$. 

\subsection{Inhomogeneous perturbation}
\label{ssec:rw_joint_inhomo}

For the inhomogeneous perturbed random walk $\bar{R}$, consider
\begin{eqnarray*}
\bar{\pi}(n_1,n_2) = c_1 \rho_1^{n_1}\sigma_1^{n_2} + c_2 \rho_2^{n_1}\sigma_2^{n_2},
\end{eqnarray*}
where $(\rho_1,\sigma_1), (\rho_2,\sigma_2) \in (0,1)^2$ are the unique points for which
\begin{eqnarray*}
(\rho_1,\sigma_1) = Q \cap H, \qquad (\rho_2,\sigma_2) = Q \cap V.
\end{eqnarray*}
Thus, besides the interior balance equations, $(\rho_1,\sigma_1)$ and $(\rho_2,\sigma_2)$ satisfy the horizontal and vertical balance equations, respectively. In Figure~\ref{fig:geometric_term_rw_joint_inhomo_0.3mu} and Figure~\ref{fig:geometric_term_rw_joint_inhomo_0.7mu}, all the curves for $R$ and the geometric terms are plotted, for cases $\mu^*=0.3\mu$ and $\mu^*=0.7\mu$. The geometric terms in $\bar{\pi}$ are marked with solid square. 

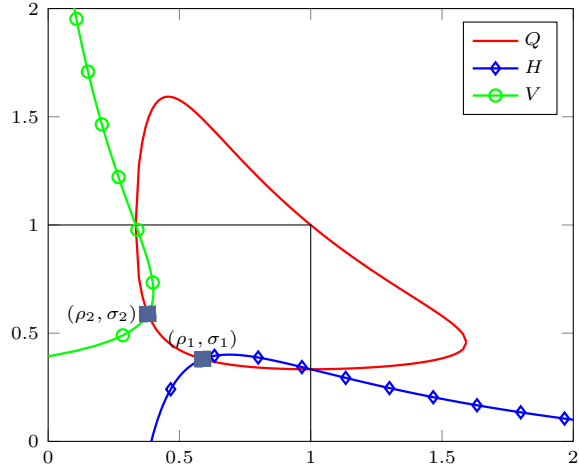
\begin{figure}[htb!]
\centering
\begin{tikzpicture}[scale=1]
\begin{axis}[
width=.7\textwidth,
xmax = 2, xmin = 0,
ymax = 2, ymin = 0,
font = \scriptsize,
legend style={
	cells={anchor=west},
	legend pos=north east,
	font = \scriptsize,
}
]

\addplot[
line width=.3mm,color=red,
mark=none
]
table[
x index=0, y index=1, col sep = comma
]
{allcurves_data_int.csv};
\addlegendentry{$Q$};

\addplot+[
forget plot,
line width=.3mm,color=red,
mark=none
]
table[
x index=0, y index=2, col sep = comma
]
{allcurves_data_int.csv};

\addplot[
line width=.3mm,color=blue, 
mark=diamond, mark repeat = 10, mark phase = 5, mark size = 0.8mm
]
table[
x index=0, y index=1, col sep = comma
]
{allcurves_data_hor.csv};
\addlegendentry{$H$};

\addplot[
line width=.3mm,color=green,  
mark=o, mark repeat = 10, mark phase = 5, mark size = 0.8mm
]
table[
x index=1, y index=0, col sep = comma
]
{allcurves_data_ver.csv};
\addlegendentry{$V$};

\addplot[
mark=none
]
coordinates {(0,1) (1,1) (1,0)};

\addplot[
mark=square*, mark size = 3,
color=darkblue,
dashed
]
coordinates {(0.5887,0.3802)} node[above,black, very thick]{$(\rho_1,\sigma_1)$};

\addplot[
mark=square*, mark size = 3, 
color=darkblue,
dashed
]
coordinates {(0.3802,0.5887)} node[left,black,very thick]{$(\rho_2,\sigma_2)$};

\end{axis}
\end{tikzpicture}
\caption{Characteristic curves for $R$ and chosen geometric terms: $\mu^*=0.3\mu$. }
\label{fig:geometric_term_rw_joint_inhomo_0.3mu}
\end{figure}
\begin{figure}[htb!]
\centering
\begin{tikzpicture}[scale=1]
\begin{axis}[
width=.7\textwidth,
xmax = 2, xmin = 0,
ymax = 2, ymin = 0,
font = \scriptsize,
legend style={
	cells={anchor=west},
	legend pos=north east,
	font = \scriptsize,
}
]

\addplot[
line width=.3mm,color=red,
mark=none
]
table[
x index=0, y index=1, col sep = comma
]
{allcurves_data_int.csv};
\addlegendentry{$Q$};

\addplot+[
forget plot,
line width=.3mm,color=red,
mark=none
]
table[
x index=0, y index=2, col sep = comma
]
{allcurves_data_int.csv};

\addplot[
line width=.3mm,color=blue, 
mark=diamond, mark repeat = 10, mark phase = 5, mark size = 0.8mm
]
table[
x index=0, y index=1, col sep = comma
]
{allcurves_data_rev_hor.csv};
\addlegendentry{$H$};

\addplot[
line width=.3mm,color=green,  
mark=o, mark repeat = 10, mark phase = 5, mark size = 0.8mm
]
table[
x index=1, y index=0, col sep = comma
]
{allcurves_data_rev_ver.csv};
\addlegendentry{$V$};

\addplot[
mark=none
]
coordinates {(0,1) (1,1) (1,0)};

\addplot[
mark=square*, mark size = 3,
color=darkblue,
dashed
]
coordinates {(0.3361,0.8764)} node[left,black,very thick]{$(\rho_1,\sigma_1)$};

\addplot[
mark=square*, mark size = 3, 
color=darkblue,
dashed
]
coordinates {(0.8764,0.3361)} node[below,black,very thick]{$(\rho_2,\sigma_2)$};

\end{axis}
\end{tikzpicture}
\caption{Characteristic curves for $R$ and chosen geometric terms: $\mu^*=0.7\mu$. }
\label{fig:geometric_term_rw_joint_inhomo_0.7mu}
\end{figure}
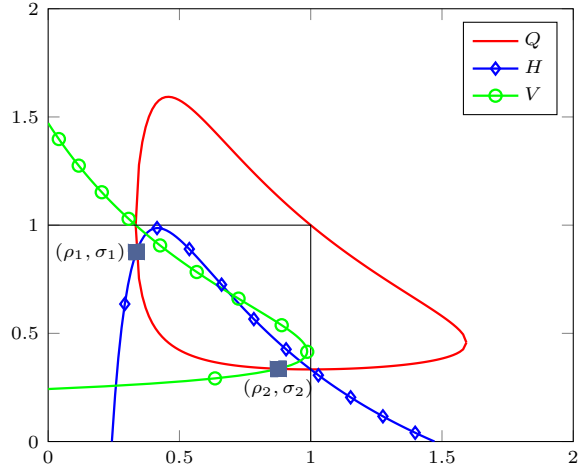

For the case $\mu^*=0.3\mu$, it can be observed that $(\rho_*,\sigma_*) = (\rho_1,\sigma_1)$, and $(\rho_{**},\sigma_{**}) = (\rho_2,\sigma_2)$. Let $\bar{h}_{1,0} = \bar{v}_{0,1} = \lambda$, and one can verify that indeed the conditions in Theorem~\ref{thm:error_bound_result} hold. Moreover, since $\rho_1 = \rho_*, \sigma_2 = \sigma_{**}$, according to Theorem~\ref{thm:limiting_probability} the limiting rates of $\bar{R}$ are equal to those of $R$, \ie
\begin{eqnarray*}
\lim_{n_1\rightarrow \infty} \bar{h}_{-1,0}(n_1,0) = \mu^*, \qquad \lim_{n_2\rightarrow \infty} \bar{v}_{0,-1}(0,n_2) = \mu^*. 
\end{eqnarray*}
In short, for the perturbed random walk $\bar{R}$, the inhomogeneous rates are approaching the original ones on the axes. However, for the case $\mu^*=0.7\mu$, $(\rho_*,\sigma_*) = (\rho_2,\sigma_2)$, and $(\rho_{**},\sigma_{**}) = (\rho_1,\sigma_1)$. Then the limiting rates are, 
\begin{eqnarray*}
\lim_{n_1\rightarrow \infty} \bar{h}_{-1,0}(n_1,0) &=& \rho_2^{-1}\lambda - (1-\rho_2)^{-1} \left( \lambda - \rho_2\sigma_2\mu \right), \\ 
\lim_{n_2\rightarrow \infty} \bar{v}_{0,-1}(0,n_2) &=& \sigma_1^{-1}\lambda - (1-\sigma_1)^{-1} \left( \lambda - \rho_1\sigma_1\mu \right),
\end{eqnarray*}
which differ from $\mu^*$. 

In the next subsection, the numerical results for error bounds in various cases are presented. 

\subsection{Numerical result for error bound}

In this part, we consider two stationary performances, the expected probability that the system is empty and the expected number of jobs in the first queue. The first performance is discussed in~\cite{goseling2016linear} and bounds on the bias terms are given explicitly there. The second performance is considered in~\cite{goseling2013energy}. For the bias terms, they are established but no explicit bounds are given. Later we shall explain that bounds on the bias terms can be obtained through a linear program. 
 
\subsubsection*{The probability of an empty system}

Consider the expected probability that the system is empty. In this case, the reward function is $F(n_1,n_2) = \mathbf{1}\{ (n_1,n_2)=(0,0) \}$. In~\cite{goseling2016linear}, it is shown that 
\begin{eqnarray*}
\left| D^t_{i,j}(n_1,n_2) \right| \leq \max\left\{ \frac{1}{\mu^*}, \frac{\mu-\mu^*}{\mu\mu^*} \right\},
\end{eqnarray*}
for all $(i,j)$ and $(n_1,n_2)$. Moreover, it is given that
\begin{eqnarray*}
\left| \bar{\mathcal{F}} - \mathcal{F} \right| \le 2\rho(1-\rho)\frac{(\mu/2-\mu^*)(\mu-\mu^*)}{\mu\mu^*}. 
\end{eqnarray*}

First, fix that $\mu^* = 0.3\mu$ and consider various loads, \ie various values for $\lambda/\mu$. The error bound results for both homogeneous and inhomogeneous perturbation are shown in Figure~\ref{fig:error_bound_load_empty_probability}, where the homogeneous one is obtained by the formula above and the inhomogeneous one is given by Theorem~\ref{thm:error_bound_result}. 

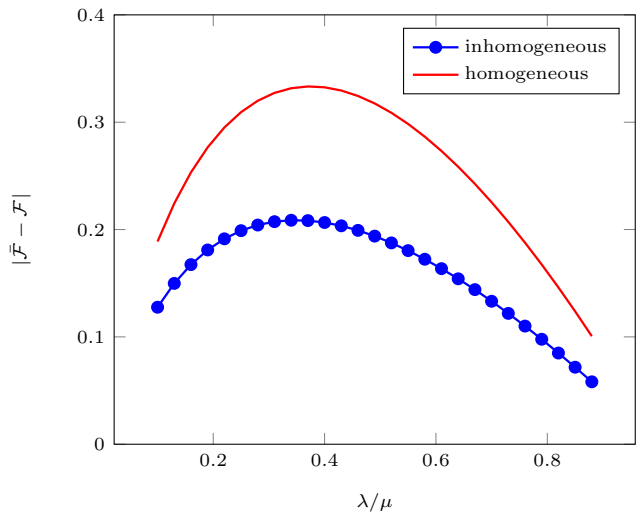
\begin{figure}[!ht]
\centering
\begin{tikzpicture}[scale=1]
\begin{axis}[
xlabel=$\lambda/\mu $,ylabel=$| \bar{\mathcal{F}} - \mathcal{F} |$, 
ymin = 0, ymax = 0.4,
font=\scriptsize,
legend style={
	cells={anchor=west},
	legend pos=north east,
	font=\scriptsize,
}
]
	
\addplot[
mark=*,line width=.3mm,color=blue
]
table[
header=false, x index=0, y index=1, col sep = comma
]
{error_bound_load_empty_probability.csv};
\addlegendentry{inhomogeneous};

\addplot[
line width=.3mm,color=red,
mark=none
]
table[
header=false,x index=0, y index=2, col sep = comma
]
{error_bound_load_empty_probability.csv};
\addlegendentry{homogeneous};	
\end{axis}     
\end{tikzpicture}
\caption{Error bound $|\bar{\mathcal{F}} - \mathcal{F}|$ for various $\lambda/\mu$: $F(n_1,n_2) = \mathbf{1}\left\{(n_1,n_2)=(0,0)\right\}$. }
\label{fig:error_bound_load_empty_probability}
\end{figure}  

As is seen from Figure~\ref{fig:error_bound_load_empty_probability}, the error bound given by inhomogeneous perturbation is better than that returned by homogeneous perturbation. The explanation behind this is that for $\bar{R}$, the inhomogeneous transition rates on the axes approach gradually to those of $R$, while the homogeneous ones remain the same distance away from the rates of $R$ everywhere. By applying the inhomogeneous perturbation, the difference between the rates of $R$ and $\bar{R}$ becomes smaller, hence the error bound is smaller. 

Next, fix that $\lambda = 0.2$, $\mu=0.6$, and let $\mu^* = \eta\cdot\mu$. In Figure~\ref{fig:error_bound_perturbation_empty_probability}, the error bounds with homogeneous and inhomogeneous perturbation for various $\eta$ are given. 
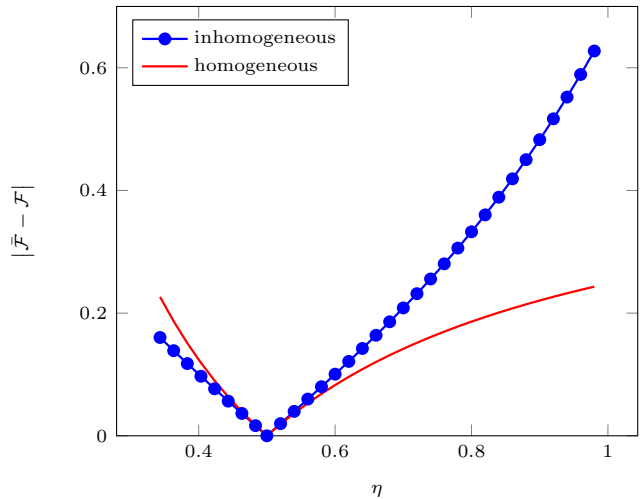
\begin{figure}[!ht]
\centering
\begin{tikzpicture}	
\begin{axis}[
xlabel=$\eta$, ylabel=$\left| \bar{\mathcal{F}}-\mathcal{F} \right|$, 
ymin = 0, ymax = 0.7,
font=\scriptsize,
legend style={
	cells={anchor=west},
	legend pos=north west,
	font=\scriptsize,
}
]
	
\addplot[
mark=*,line width=.3mm,color=blue
]
table[
header=false, x index=0, y index=1, col sep = comma
]
{error_bound_perturbation_empty_probability.csv};
\addlegendentry{inhomogeneous};
	
\addplot[
line width=.3mm,color=red,
mark=none
]
table[
header=false,x index=0, y index=2, col sep = comma
]
{error_bound_perturbation_empty_probability.csv};
\addlegendentry{homogeneous};	
\end{axis}     
\end{tikzpicture}
\caption{Error bound $|\bar{\mathcal{F}} - \mathcal{F}|$ for various $\eta$: $F(n_1,n_2) = \mathbf{1}\left\{(n_1,n_2)=(0,0)\right\}$. }
\label{fig:error_bound_perturbation_empty_probability}
\end{figure}
It can be seen that the inhomogeneous perturbation does not always give better error bound than homogeneous one. When $\eta = 0.5$, $R$ itself has a product-form stationary distribution hence the error bound is $0$. If $\eta < 0.5$, from the previous description, it is known that 
\begin{eqnarray*}
\lim_{n_1\rightarrow \infty} \bar{h}_{-1,0}(n_1,0) = h_{-1,0}, \qquad \lim_{n_2\rightarrow \infty} \bar{v}_{0,-1}(0,n_2) = v_{0,-1}. 
\end{eqnarray*}
Because of the argument before, the inhomogeneous perturbation outperforms the homogeneous one. Moreover, if $\eta$ is small, there is a relatively big difference between transition rates of $R$ and $\bar{R}$ for homogeneous perturbation while the rates of $\bar{R}$ are approaching those of $R$ gradually for inhomogeneous perturbation. Therefore, the difference between homogeneous and inhomogeneous perturbation gets bigger when $\eta$ becomes smaller. However, if $\eta > 0.5$, the limiting rates are further away from the original ones for inhomogeneous perturbation than the homogeneous one. Hence, the error bound for inhomogeneous perturbation is worse than that for the homogeneous one. In the next part, the same model is considered with a different performance measure. 

\subsubsection*{The expected number of jobs in the first queue}

Consider the expected number of jobs in the first queue. In this case the reward function is $F(n_1,n_2) = n_1$. In~\cite{goseling2013energy}, it is shown that the bias terms $D_{i,j}^t(n_1,n_2)$ can be either non-negative or non-positive depending on $(n_1,n_2)$ and $(i,j)$. But there is no explicit expression available for the bounds on the bias terms. Following from the approach in~\cite{goseling2016linear}, we formulate a linear program in which the error bound result given by Theorem~\ref{thm:error_bound_result} is the objective function and the coefficients of $B_1(n_1,0)$, $B_2(0,n_2)$ and $B_3$ are the variables. By minimizing the error bound, bounding functions can be found for the bias terms. 

Again, first fix $\mu^* = 0.3\mu$ and consider various loads. For each case, bounding functions are found by solving the linear program. These functions are then used to get the error bound result. The error bounds obtained by homogeneous and inhomogeneous perturbation are shown in Figure~\ref{fig:error_bound_load_num_first_queue}. 
\begin{figure}[!ht]
\centering
\begin{tikzpicture}[scale=1]
\begin{axis}[
xlabel=$\lambda/\mu $,ylabel=$| \bar{\mathcal{F}} - \mathcal{F} |$, 
ymin = 0, ymax = 11,
font=\scriptsize,
legend style={
	cells={anchor=west},
	legend pos=north west,
	font=\scriptsize,
}
]

\addplot[
mark=*, mark size=0.5mm,line width=.3mm,color=blue
]
table[
header=false, x index=0, y index=1, col sep = comma
]
{error_bound_load_num_first_queue.csv};
\addlegendentry{inhomogeneous};

\addplot[
line width=.3mm,color=red,
mark=none
]
table[
header=false,x index=0, y index=2, col sep = comma
]
{error_bound_load_num_first_queue.csv};
\addlegendentry{homogeneous};	
\end{axis}     
\end{tikzpicture}
\caption{Error bound $|\bar{\mathcal{F}} - \mathcal{F}|$ with homogeneous and inhomogeneous perturbation for various $\lambda/\mu$: $F(n_1,n_2) = n_1$. }
\label{fig:error_bound_load_num_first_queue}
\end{figure}
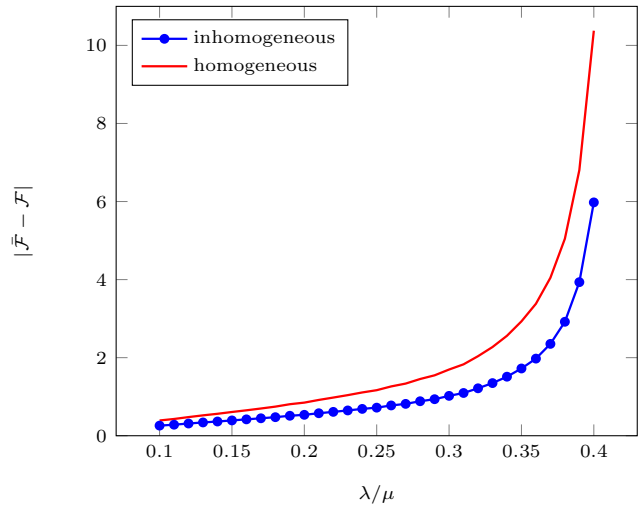
It can be observed that the inhomogeneous perturbation performs better than the homogeneous one, since the inhomogeneous rates are approaching those of $R$ in the limit. When $\lambda/\mu > 0.4$, similar to the observation in~\cite{goseling2016linear}, the linear programming problem does not have a solution satisfying all the constraints. Hence, the corresponding part is not included in the figure. 

Next, let $\lambda = 0.2$, $\mu = 0.6$, and take $\mu^* = \eta \cdot \mu$. The error bound results for homogeneous and inhomogeneous perturbation are also given below. It is seen that, due to the same reason, the inhomogeneous perturbation is better than the homogeneous one when $\eta < 0.5$. If $\eta > 0.5$, the error bound for homogeneous perturbation is smaller. 

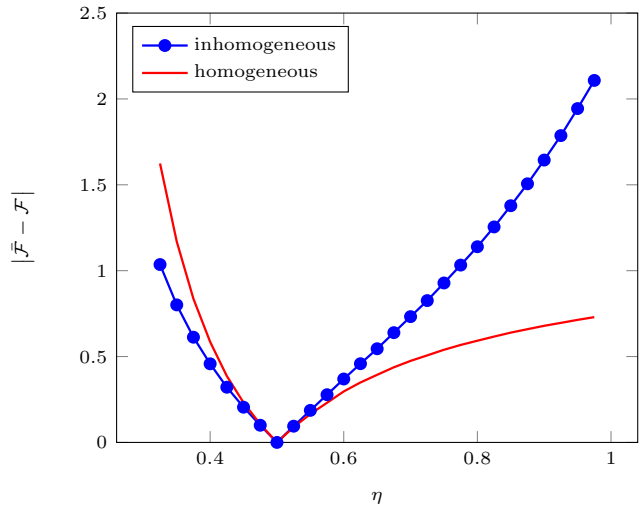
\begin{figure}[!ht]
\centering
\begin{tikzpicture}	
\begin{axis}[
xlabel=$\eta$, ylabel=$\left| \bar{\mathcal{F}}-\mathcal{F} \right|$, 
ymin = 0, ymax = 2.5,
font=\scriptsize,
legend style={
	cells={anchor=west},
	legend pos=north west,
	font=\scriptsize,
}
]

\addplot[
mark=*,line width=.3mm,color=blue
]
table[
header=false, x index=0, y index=1, col sep = comma
]
{error_bound_perturbation_num_first_queue.csv};
\addlegendentry{inhomogeneous};

\addplot[
line width=.3mm,color=red,
mark=none
]
table[
header=false,x index=0, y index=2, col sep = comma
]
{error_bound_perturbation_num_first_queue.csv};
\addlegendentry{homogeneous};	
\end{axis}     
\end{tikzpicture}
\caption{Error bound $|\bar{\mathcal{F}} - \mathcal{F}|$ with homogeneous and inhomogeneous perturbation for various $\eta$: $F(n_1,n_2) = n_1$. }
\end{figure}

\section{Discussion}
\label{sec:discussion}

In this paper, an inhomogeneous perturbation framework is considered for the expected stationary reward of a continuous-time random walk. For a given probability distribution that is a sum of geometric terms, an approach to construct the inhomogeneous transition rates on the boundaries of the state space is proposed for the perturbed random walk. Indeed, with the construction, the given distribution is the stationary probability distribution of the perturbed random walk. Moreover, an explicit expression for the error bound is given based on the proposed inhomogeneous perturbation. Numerical results demonstrate that inhomogeneous perturbation can provide smaller error bounds than homogeneous perturbation. As is seen from the numerical results, the quality of the error bound depends on the choice of the parameters used in $\bar\pi$. In fact, in some cases the best choice seems to be to use a single geometric term, i.e. product-form distribution, and perform a homogeneous perturbation. Part of future work is to optimize over the choice of $(\rho_k.\sigma_k)$ as well as $c_k$ in the given probability distribution.

\subsubsection*{Acknowledgments}
Xinwei Bai acknowledges support by a CSC scholarship [No. 201407720012].

\bibliographystyle{plain}
\bibliography{references}

\begin{thebibliography}{10}

\bibitem{adan1993compensation}
I.~J. B.~F. Adan, J.~Wessels, and W.~H.~M. Zijm.
\newblock A compensation approach for two-dimensional markov processes.
\newblock {\em Advances in Applied Probability}, pages 783--817, 1993.

\bibitem{altman2004perturbation}
E.~Altman, K.~E. Avrachenkov, and R.~N{\'u}{\~n}ez-Queija.
\newblock Perturbation analysis for denumerable markov chains with application
  to queueing models.
\newblock {\em Advances in Applied Probability}, 36(03):839--853, 2004.

\bibitem{chen2013necessary}
Y.~Chen, R.~J. Boucherie, and J.~Goseling.
\newblock Necessary conditions for the invariant measure of a random walk to be
  a sum of geometric terms.
\newblock {\em arXiv preprint arXiv:1304.3316}, 2013.

\bibitem{chen2015invariant}
Y.~Chen, R.~J. Boucherie, and J.~Goseling.
\newblock Invariant measures and error bounds for random walks in the
  quarter-plane based on sums of geometric terms.
\newblock {\em arXiv preprint arXiv:1502.07218}, 2015.

\bibitem{fayolle1999random}
G.~Fayolle, R.~Iasnogorodski, and V.~A. Malyshev.
\newblock {\em Random walks in the quarter-plane: algebraic methods, boundary
  value problems and applications}, volume~40.
\newblock Springer Science \& Business Media, 1999.

\bibitem{goseling2013energy}
J.~Goseling, R.~J. Boucherie, and J.~C.~W. van Ommeren.
\newblock Energy--delay tradeoff in a two-way relay with network coding.
\newblock {\em Performance Evaluation}, 70(11):981--994, 2013.

\bibitem{goseling2016linear}
Jasper Goseling, Richard~J Boucherie, and Jan-Kees van Ommeren.
\newblock A linear programming approach to error bounds for random walks in the
  quarter-plane.
\newblock {\em Kybernetika}, 52(5):757--784, 2016.

\bibitem{haviv1984perturbation}
M.~Haviv and L.~van~der Heyden.
\newblock Perturbation bounds for the stationary probabilities of a finite
  markov chain.
\newblock {\em Advances in Applied Probability}, 16(04):804--818, 1984.

\bibitem{heidergott2010series}
B.~Heidergott, A.~Hordijk, and N.~Leder.
\newblock Series expansions for continuous-time markov processes.
\newblock {\em Operations Research}, 58(3):756--767, 2010.

\bibitem{heidergott2007series}
B.~Heidergott, A.~Hordijk, and M.~Van~Uitert.
\newblock Series expansions for finite-state markov chains.
\newblock {\em Probability in the Engineering and Informational Sciences},
  21(03):381--400, 2007.

\bibitem{miyazawa2011light}
M.~Miyazawa.
\newblock Light tail asymptotics in multidimensional reflecting processes for
  queueing networks.
\newblock {\em Top}, 19(2):233--299, 2011.

\bibitem{vandijk1988simple}
N.~M. van Dijk.
\newblock Simple bounds for queueing systems with breakdowns.
\newblock {\em Performance Evaluation}, 8(2):117--128, 1988.

\bibitem{vandijk1998bounds}
N.~M. van Dijk.
\newblock Bounds and error bounds for queueing networks.
\newblock {\em Annals of Operations Research}, 79:295--319, 1998.

\bibitem{vandijk11inbook}
N.~M. Van~Dijk.
\newblock Error bounds and comparison results: The markov reward approach for
  queueing networks.
\newblock In R.~J. Boucherie and N.~M. Van~Dijk, editors, {\em Queueing
  Networks: A Fundamental Approach}, volume 154 of {\em International Series in
  Operations Research \& Management Science}. Springer, 2011.

\bibitem{vandijk1988perturbation}
N.~M. van Dijk and M.~L. Puterman.
\newblock Perturbation theory for markov reward processes with applications to
  queueing systems.
\newblock {\em Advances in applied probability}, 20(01):79--98, 1988.

\end{thebibliography}

\end{document}